\documentclass[a4paper,11pt]{article}
\usepackage{amsfonts,amssymb,amsmath, dsfont,relsize,accents}
\usepackage{bm}
\usepackage{mathrsfs}  
\usepackage[usenames]{color}
\usepackage{upgreek}
\usepackage[english]{babel}
\usepackage{graphicx}
\usepackage{microtype}
\usepackage[colorlinks=true, pdfstartview=FitV, linkcolor=blue, citecolor=blue, urlcolor=blue,pagebackref=false]{hyperref}
\usepackage{graphicx}
\usepackage{epstopdf}
\definecolor{labelkey}{gray}{.8}
\definecolor{refkey}{gray}{.8}

\definecolor{darkred}{rgb}{0.9,0.1,0.1}

 \makeatletter
 \@addtoreset{equation}{section}
 \makeatother

 \makeatletter
 \@addtoreset{enunciato}{section}
 \makeatother

 \newcounter{enunciato}[section]

 \newtheorem{ittheorem}{Theorem}
 \newtheorem{itlemma}{Lemma}
 \newtheorem{itproposition}{Proposition}
 \newtheorem{itcorollary}{Corollary}
 \newtheorem{itdefinition}{Definition}
 \newtheorem{itremark}{Remark}
 \newtheorem{itclaim}{Claim}
 \newtheorem{itfact}{Fact}
 \newtheorem{itconjecture}{Conjecture}

 \newenvironment{theorem}{\addtocounter{enunciato}{1}
 \begin{ittheorem}}{\end{ittheorem}}

 \newenvironment{lemma}{\addtocounter{enunciato}{1}
 \begin{itlemma}}{\end{itlemma}}

 \newenvironment{proposition}{\addtocounter{enunciato}{1}
 \begin{itproposition}}{\end{itproposition}}

 \newenvironment{corollary}{\addtocounter{enunciato}{1}
 \begin{itcorollary}}{\end{itcorollary}}

 \newenvironment{definition}{\addtocounter{enunciato}{1}
 \begin{itdefinition}}{\end{itdefinition}}

 \newenvironment{remark}{\addtocounter{enunciato}{1}
 \begin{itremark}}{\end{itremark}}

 \newenvironment{claim}{\addtocounter{enunciato}{1}
 \begin{itclaim}}{\end{itclaim}}

 \newenvironment{fact}{\addtocounter{enunciato}{1}
 \begin{itfact}}{\end{itfact}}

 \newenvironment{conjecture}{\addtocounter{enunciato}{1}
 \begin{itconjecture}}{\end{itconjecture}}

 \newcommand{\be}[1]{\begin{equation}\label{#1}}
 \newcommand{\ee}{\end{equation}}

 \newcommand{\bl}[1]{\begin{lemma}\label{#1}}
 \newcommand{\el}{\end{lemma}}

 \newcommand{\br}[1]{\begin{remark}\label{#1}}
 \newcommand{\er}{\end{remark}}

 \newcommand{\bt}[1]{\begin{theorem}\label{#1}}
 \newcommand{\et}{\end{theorem}}

 \newcommand{\bd}[1]{\begin{definition}\label{#1}}
 \newcommand{\ed}{\end{definition}}

 \newcommand{\bcl}[1]{\begin{claim}\label{#1}}
 \newcommand{\ecl}{\end{claim}}

 \newcommand{\bfact}[1]{\begin{fact}\label{#1}}
 \newcommand{\efact}{\end{fact}}

 \newcommand{\bp}[1]{\begin{proposition}\label{#1}}
 \newcommand{\ep}{\end{proposition}}

 \newcommand{\bc}[1]{\begin{corollary}\label{#1}}
 \newcommand{\ec}{\end{corollary}}

 \newcommand{\bcj}[1]{\begin{conjecture}\label{#1}}
 \newcommand{\ecj}{\end{conjecture}}

 \newcommand{\bpr}{\begin{proof}}
 \newcommand{\epr}{\end{proof}}

 \newcommand{\bprlem}[1]{\begin{proofof}{\it Lemma \ref{#1}}.\,\,}
 \newcommand{\eprlem}{\end{proofof}}

 \newcommand{\bprthm}[1]{\begin{proofof}{\it Theorem \ref{#1}}.\,\,}
 \newcommand{\eprthm}{\end{proofof}}

 \newcommand{\bprprop}[1]{\begin{proofof}{\it Proposition \ref{#1}}.\,\,}
 \newcommand{\eprprop}{\end{proofof}}

 \newcommand{\bi}{\begin{itemize}}
 \newcommand{\ei}{\end{itemize}}

 \newcommand{\ben}{\begin{enumerate}}
 \newcommand{\een}{\end{enumerate}}

 \newenvironment{proof}{\noindent {\em Proof}.\,\,}{\hspace*{\fill}$\halmos$\medskip}
 \newenvironment{proofof}{\noindent {\em Proof of\,\,}}{\hspace*{\fill}$\halmos$\medskip}
 \newcommand{\halmos}{\rule{1ex}{1.4ex}}

 \newcommand{\one}{{\mathchoice {1\mskip-4mu\mathrm l}
         {1\mskip-4mu\mathrm l}
         {1\mskip-4.5mu\mathrm l}
         {1\mskip-5mu\mathrm l}}}

\def \E {{\mathbb E}}

\def \N {{\mathbb N}}
\def \P {{\mathbb P}}

\def \R {{\mathbb R}}
\def \Z {{\mathbb Z}}

\def \lra \leftrightarrow

\def \ra {\rightarrow}
\def \ba {\begin{array}}
\def \ea {\end{array}}

\def \lra {\longrightarrow}

\def \lra {{\leftrightarrow}}

\def \subset {\subseteq}

\def \emptyset {\varnothing}

\def\one{\rlap{\mbox{\small\rm 1}}\kern.15em 1}

\newlength{\dhatheight}

\begin{document}

\title{On the  threshold of spread-out voter model percolation}

\author{Bal\'azs R\'ath\textsuperscript{1}, Daniel Valesin\textsuperscript{2}}
\footnotetext[1]{Budapest University of Technology and Economics, Egry J\'ozsef u. 1, 1111 Budapest, Hungary. \\ \url{rathb@math.bme.hu}}
\footnotetext[2]{\noindent University of Groningen, Nijenborgh 9, 9747 AG Groningen, The Netherlands.\\ \url{d.rodrigues.valesin@rug.nl}}
\date{May 16, 2017}
\maketitle

\begin{abstract}
In the $R$-spread out, $d$-dimensional voter model, each site $x$ of $\mathbb{Z}^d$ has state (or `opinion') 0 or 1 and, with rate 1, updates its opinion by copying that of some site $y$ chosen uniformly at random among all sites within distance $R$ from $x$. If $d \geq 3$, the set of (extremal) stationary measures of this model is given by a family $\mu_{\alpha, R}$, where $\alpha \in [0,1]$. Configurations sampled from this measure are polynomially correlated fields of 0's and 1's in which the density of 1's is $\alpha$ and the correlation weakens as $R$ becomes larger. We study these configurations from the point of view of nearest neighbor site percolation on $\mathbb{Z}^d$, focusing on asymptotics as $R \to \infty$. In \cite{RV15}, we have shown that, if $R$ is large, there is a critical value $\alpha_c(R)$ such that there is percolation if $\alpha > \alpha_c(R)$ and no percolation if $\alpha < \alpha_c(R)$. Here we prove that, as $R \to \infty$, $\alpha_c(R)$ converges to the critical probability for Bernoulli site percolation on $\mathbb{Z}^d$. Our proof relies on a new upper bound on the joint occurrence of events under $\mu_{\alpha,R}$ which is of independent interest.
\end{abstract}


\section{Introduction}
The \textit{voter model} on $\Z^d$ with range $R \in \N$ is a Markov process $(\xi_t)_{t\geq 0}$ on $\{0,1\}^{\Z^d}$ with infinitesimal pregenerator defined as follows, for any function $f: \{0,1\}^{\Z^d} \to \R$  that depends only on finitely many coordinates:
\begin{equation}(\mathcal{L}f)(\xi) = \sum_{\substack{x,y \in \Z^d:\\0<|x-y|_1 \leq R}}\frac{f(\xi^{y\to x}) - f(\xi)}{|B_1(R)|-1},  \label{eq:generator}\end{equation}
where  $|\cdot |_1$ is the $\ell^1$-norm on $\Z^d$, $B_1(R)$ is the set of vertices of $\Z^d$ with $\ell^1$-norm smaller than or equal to $R$, $|B_1(R)|$ is the cardinality of this set and
$$\xi^{y\to x}(z) = \begin{cases} \xi(z),&\text{if } z \neq x,\\\xi(y),&\text{if }z = x,\end{cases}\quad z \in \Z^d. $$
In the usual interpretation, sites of $\Z^d$ represent individuals (``voters'') and the states 0 and 1 represent two conflicting opinions. The dynamics defined by \eqref{eq:generator} is then explained in words as follows. Individuals are all endowed with independent exponential clocks (all with parameter 1); whenever the clock of individual $x$ rings, another individual $y$ is chosen uniformly at random within $\ell^1$-distance at most $R$ from $x$, and then $x$ copies the opinion of $y$.

This process has been introduced independently in \cite{CS73} and \cite{HL75}. We refer the reader to \cite{Li85} for the general theory on the voter model, including all statements that we mention without explicit reference in this introduction.

Let $\mathscr{I}_{d,R}$ denote the set of extremal stationary distributions of the voter model on $\Z^d$ and range $R$. In case $d = 1$ or $2$, for any $R$, this set consists only of $\delta_{\underline{0}}$ and $\delta_{\underline{1}}$, the two measures that give full mass to the configurations which are identically equal to 0 or 1, respectively. In case $d \geq 3$, $\mathscr{I}_{d,R}$ consists of a one-parameter family of measures
$$\{\mu_{\alpha, R}:\; 0 \leq \alpha \leq 1\}$$
(we will generally omit the dimension $d$ from our notation). For each $\alpha$, $\mu_{\alpha, R}$ is obtained as the distributional limit as time is taken to infinity (which is shown to exist) of the process started from the measure in which the states at all sites are independent and distributed as Bernoulli($\alpha$). Each of the measures $\mu_{\alpha,R}$ is invariant and ergodic with respect to translations in $\Z^d$. Additionally,
$$\mu_{\alpha, R}(\{\xi: \xi(0) = 1\}) = \alpha,$$
so that $\alpha$ is a density parameter. Finally, $\mu_{\alpha,R}$ exhibits polynomial decay of correlations:  for any $d \geq 3$, $R \in \N$ and $\alpha \in (0,1)$,
\begin{equation}
c(\alpha,R)\cdot |x-y|_1^{2-d} < \text{Cov}_{\mu_{\alpha,R}}(\xi(x),\xi(y)) <C(\alpha,R)\cdot |x-y|_1^{2-d},\quad x,y\in \Z^d, \; x\neq y.
\end{equation}

In \cite{RV15}, addressing earlier work by \cite{LS86,BLM87,ML06,Ma07}, the authors have considered the problem of \textit{percolation phase transition} of the measure $\mu_{\alpha, R}$, which will now be enunciated. For given values of $d$, $R$ and $\alpha$, let $\xi \in \{0,1\}^{\Z^d}$ be a configuration sampled from $\mu_{\alpha,R}$. Consider the subgraph of the nearest-neighbor lattice $\Z^d$ induced by the set of vertices $\{x: \xi(x) = 1\}$ (i.e., the set of \emph{open} sites). Let $\text{Perc}$ be the event that this subgraph contains an infinite connected component (\textit{cluster}).
 By ergodicity, $\mu_{\alpha,R}(\text{Perc})$ is either 0 or 1.  The statement that the measures $\mu_{\alpha, R}$ exhibit a non-trivial percolation phase transition with respect to the density parameter $\alpha$ means that, for any $d \geq 3$ and $R \in \N$, there exists $\alpha_c = \alpha_c(R) \in (0,1)$ (depending on $d$ and $R$) such that $\mu_{\alpha, R}(\text{Perc})=0$ if $\alpha < \alpha_c$ and $\mu_{\alpha, R}(\text{Perc})=1$ is $\alpha > \alpha_c$. The main result of \cite{RV15} is that this is indeed the case under two sets of assumptions: first, $d \geq 5$, and second, $d = 3$ or $4$ and $R$ large enough.

In the present paper, we continue this investigation by considering the percolation event under $\mu_{\alpha,R}$ when $R$ is taken to infinity. Before stating our result, we make a brief detour which will make the statement more natural. Let us first present a well-known alternate construction of $\mu_{\alpha, R}$ through \textit{coalescing random walks}.

Consider a collection of particles simultaneously performing random walks on $\Z^d$ and subject to the following rules. At time 0, each site of $\Z^d$ contains one particle. Each particle decides to jump to a new location after an amount of time distributed as Exponential(1). Jumping from a site $x$, a particle chooses its destination $y$ uniformly at random among all sites of $B_1(x,R)\setminus \{ x \}$. If $y$ is already occupied by another particle, the two particles coalesce, becoming a single particle. 

This process, when run for all times $0 \leq t < \infty$, induces a partition of $\Z^d$ as follows. We say that $x,y \in \Z^d$ are in the same partition class if the particle at $x$ at time 0 eventually coalesces with the particle at $y$ at time 0 (by this we include situations in which these particles coalesce with other particles before coalescing with each other). Note that almost surely each partition class has infinite cardinality.
 See Section \ref{section:coal} for more precise definitions. Given these partition classes, we then independently assign 0's and 1's to each class with probability $\alpha$ and $1-\alpha$, respectively. The distribution of the resulting configuration $\xi \in \{0,1\}^{\Z^d}$ then coincides with $\mu_{\alpha, R}$. This construction is a consequence of the well-known fact that the voter model and the system of coalescing random walks just described exhibit a temporal \textit{duality} relation.

Assume $d \geq 3$ and $A = \{x_1,\ldots, x_n\} \subset \Z^d$ be an arbitrary finite set. In case $R$ is very large (compared for example to the diameter of $A$), it is very likely that the particles initially located at $A$ will quickly disperse and never coalesce with each other. Indeed, the probability that two range-$R$ random walks on $\Z^d$ ever meet tends to zero as $R \to \infty$, uniformly over their initial locations, see \eqref{eq:bound_green} below. Thus, all the particles initially at $A$ will end up in distinct partition classes, so that, for any $n$ and any $(i_1,\ldots, i_n) \in \{0,1\}^n$ and any $n$-tuple $(x_1,\ldots,x_n)$ of distinct vertices in $\Z^d$ we have
\begin{equation}\label{local_conv_to_bernoulli}
\lim_{R \to \infty}  \mu_{\alpha,R}\big( \, (\xi(x_1),\ldots, \xi(x_n)) = (i_1,\ldots, i_n)\, \big) = \alpha^{\sum_k i_k}\cdot (1-\alpha)^{n - \sum_k i_k}.
\end{equation}
Another way of stating this is that, as $R\to\infty$, $\mu_{\alpha,R}$ converges weakly (taking the product topology on the space of configurations) to $\pi_\alpha$, the infinite product over $\Z^d$ of the Bernoulli($\alpha$) distribution. Let
$$p_c = \sup\{ \; p: \pi_p(\text{Perc}) = 0 \; \}$$
be the critical parameter of independent Bernoulli site percolation in $\Z^d$; see \cite{Gr99} for the well-known fact that $p_c \in (0,1)$ for any $d \geq 2$. We are now ready to state our main result.
\begin{theorem}\label{thm:main} For any $d\geq 3$, as $R \to \infty$, the critical density value for percolation phase transition of the stationary measures of the voter model with range $R$ converges to the critical density value for independent Bernoulli percolation:
\begin{equation}
\lim_{R\to\infty} \alpha_c(R) = p_c.
\end{equation}
\end{theorem}

This convergence result seems natural given \eqref{local_conv_to_bernoulli}, but the proof is not at all
automatic, as we now argue. First, $\mu_{\alpha,R}$ cannot be stochastically dominated (or minorated)
by a Bernoulli product measure $\pi_p, p \in (0,1)$: the $R=1$ case is proved in 
 \cite[Section 5.3.2]{st17} (using results of \cite{LS88}), the proof for general $R$ is identical.
 Second,  just because a probability measure on $\{0,1\}^{\Z^d}$ ``locally" looks like Bernoulli percolation, we
cannot automatically draw any conclusions about the percolative properties of open sites:
 \cite[Theorem 19]{bgp}  states that for any $K$ and any $p \in (0,1)$, there exists a
 probability measure $\mu$ on $\{0,1\}^{\Z^d}$ that satisfies 
 \begin{equation}\label{k_wise_indep}
 \mu\big( \, (\xi(x_1),\ldots, \xi(x_K)) = (i_1,\ldots, i_K)\, \big) = \alpha^{\sum_k i_k}\cdot (1-\alpha)^{K - \sum_k i_k}
\end{equation}
 for  any $(i_1,\ldots, i_K) \in \{0,1\}^K$ and any $K$-tuple $(x_1,\ldots,x_K)$ of distinct vertices in $\Z^d$ such that $\mu(\text{Perc})=1$, moreover
 there exists another $\mu$ satisfying \eqref{k_wise_indep} for which $\mu(\text{Perc})=0$.

 \textit{Decoupling inequalities} are often instrumental in dealing with polynomially correlated percolation models. Powerful such inequalities have been proved for other models, such as random interlacements \cite{sznitman_decoupling, popov_teixeira}, the Gaussian free field \cite{RS13, popov_rath}  and certain massless gradient Gibbs measures \cite{rodriguez17}, using the so-called \emph{sprinkling} technique. In contrast, our main tool is Lemma \ref{lem:disjointly}, which is not a decoupling inequality and is not proved through sprinkling. 
  It allows for a direct comparison between the measures $\mu_{\alpha,R}$ and $\pi_\alpha$ -- and hence a direct control on correlations present in $\mu_{\alpha,R}$ -- which relies on a natural coupling of systems of independent random walks, coalescing random walks and annihilating random walks.

\section{Notation and preliminary results}

\subsection{Notation for sets, paths and configurations}
For any set $S$, the cardinality of $S$ is denoted by $|S|$ and the indicator function of $S$ by $\mathds{1}_S$.

For a vector $x$ in $\mathbb{Z}^d$, the $\ell^\infty$-norm of $x$ is denoted by $|x|$ and the $\ell^1$-norm of $x$ by $|x|_1$. Two vertices $x,y$ are \textit{neighbors} if $|x-y|_1 = 1$; we denote this by $x \sim y$. Vertices $x$ and $y$ are $*$-neighbors if $|x-y| = 1$.

The balls and spheres corresponding to these norms are then given by
\begin{align*}
&B(L) = \{x \in \Z^d: |x| \leq L\}, &B(x,L) = \{y \in \Z^d: |x-y| \leq L\},\\
&B_1(L) = \{x \in \Z^d: |x|_1 \leq L\}, &B_1(x,L) = \{y \in \Z^d: |x-y|_1 \leq L\},\\
&S(L) = \{x \in \Z^d: |x|= L\}, &S(x,L) = \{y \in \Z^d: |x-y|= L\}.
\end{align*}

Given a finite set $A \subset \Z^d$, the \textit{diameter} of $A$ is 
\begin{equation*}
\text{diam}(A) = \sup\{|x-y|: x,y \in A\}. \end{equation*}
Given sets $A, B \subset \Z^d$, the \textit{distance} between $A$ and $B$ is
\begin{equation*}
\text{dist}(A,B) = \min\{|x-y|:x \in A,\;y \in B\}.
\end{equation*}

A \textit{nearest-neighbor} path in $\Z^d$ is a finite or infinite sequence $\gamma = (\gamma(0),\gamma(1),\ldots)$ such that $\gamma(i)\sim \gamma(i+1)$ for each $i$. A \textit{$*$-connected path} is a sequence $\gamma = (\gamma(0), \gamma(1),\ldots)$ such that $\gamma(i)$ and $\gamma(i+1)$ are $*$-neighbors for each $i$. We observe that any nearest-neighbor path is also a $*$-connected path.

Given disjoint sets $A, B \subset \Z^d$ and a configuration $\xi \in \{0,1\}^{\Z^d}$, we say that $A$ and $B$ are connected by an open path in $\xi$ (and write $A \stackrel{\xi}{\leftrightarrow} B$) if there exists a nearest-neighbor path $\gamma = (\gamma(0), \ldots, \gamma(n))$ such that $\gamma(0)$ is the neighbor of a point of $A$, $\gamma(n)$ is the neighbor of a point of $B$ and $\xi(\gamma(i)) = 1$ for all $i$. Similarly, we write $A\stackrel{*\xi}{\leftrightarrow} B$ if there exists a $*$-connected path from a $*$-neighbor of a point of $A$ to a $*$-neighbor of a point of $B$ and $\xi$ is equal to 1 at all points in this path.

 The collection of cylinder sets of $\{0,1\}^{\Z^d}$ associated to $A$ is denoted by $\mathcal{F}_A$. This is the set of subsets of $\{0,1\}^{\Z^d}$ of the form
\begin{equation} \label{eq:cylinder}\{\xi \in \{0,1\}^{\Z^d}: \xi|_A \in E_0\},\end{equation}
where $E_0 \subseteq \{0,1\}^A$ and $\xi|_A$ denotes the restriction of $\xi$ to $A$. Sometimes, as an abuse of notation, the set in \eqref{eq:cylinder} and the corresponding set $E_0$ will be treated as if they were the same.
As usual, we endow $\{0,1\}^{\Z^d}$ with the $\sigma$-algebra $\mathcal{F}$ generated by all the cylinder sets.

For $x \in \Z^d$ and $\xi\in \{0,1\}^{\Z^d}$, we let $\uptau_x\xi$ be the configuration given by
$$(\uptau_x\xi)(y) = \xi(y-x),\quad y \in \Z^d.$$
Given $E \in \mathcal{F}$, let $\theta_x E = \{\uptau_x\xi: \xi \in E\}$. In particular, if $E \in \mathcal{F}_A$, then $\theta_x E \in \mathcal{F}_{A + x}$.

For $\alpha \in [0,1]$, we denote by $\pi_\alpha$ the product Bernoulli($\alpha$) measure on $\mathcal{F}$.

Let $\prec$ denote a well-ordering of $\Z^d$.

\subsection{Spread-out random walk}
We call an \textit{$R$-spread out random walk} on $\Z^d$ started at $z \in \Z^d$ the continuous-time Markov chain $(X^z_t)_{t\geq 0}$ on $\Z^d$ with $X^z(0) = z$,  Exponential(1) holding times which jumps from any site $x$ to a site uniformly chosen in $B_R(x)$. Its infinitesimal generator is thus
$$(Lf)(x)  = \sum_{\substack{y \in \Z^d:\\0 < |x-y|_1 \leq R}} \frac{f(y) - f(x)}{|B_1(R)|-1},$$
with $f: \Z^d \to \R$. 
Given distinct vertices $x,y \in \Z^d$, assume $(X^x_t)$ and $(X^y_t)$ are independent $R$-spread out random walks started at $x$ and $y$ (and let $\mathbf{P}$ be a probability measure under which these are defined). We then let
\begin{equation}\label{eq:def_h}
h_R(x,y) = \mathbf{P}\left[\, \exists t:\; X^x_t = X^y_t \, \right]
\end{equation}
be the probability that these walks ever meet. Claim 2.7 in \cite{RV15} states that
\begin{equation}\label{eq:bound_green}
\forall R > 0,\; x,y \in \Z^d,\;x\neq y,\quad h_R(x,y) \leq f(R)\cdot |x-y|^{2-d},\; \lim_{R \to \infty} f(R) = 0.
\end{equation}


\subsection{Coalescing and annihilating random walks}
\label{section:coal}

In this section  we present a construction of systems of coalescing random walks on $\Z^d$, which, as explained in the Introduction, are used to obtain the measures $\mu_{\alpha,R}$. A typical \emph{graphical} construction of coalescing random walks consists of Poisson processes dictating jump times of particles; see for
 instance  \cite[Section 3]{RV15}. Here we rely on a different approach, using an auxiliary process which we call a process of \textit{marked partitions}. This approach is intuitively appealing and quite convenient for our proofs; in particular, it allows for a useful coupling of systems of independent, coalescing and annihilating random walks. 
 
On the other hand, the marked partition approach has the drawback of only being suitable for systems consisting of finitely many particles; that is, we start by fixing a finite set $A \subset \Z^d$ and define the system of coalescing walks in which, at time 0, there is one particle in each vertex of $A$. This will allow us to obtain the projection of $\mu_{\alpha,R}$ to $A$, which is sufficient for our purposes. 

\medskip

Let $A$ be a finite subset of $\Z^d$; this set will be fixed throughout Section \ref{section:coal}. A \textit{marked partition} of $A$ is a partition of $A$ into
\emph{blocks} (i.e., subsets) together with a set of marked vertices (i.e., distinguished vertices of $A$) such that each block of the partition contains exactly one marked vertex. We represent a marked partition by $\Pi = (M, \ell)$, where $M \subset A$ is the set of marks and $\ell: A \to M$ is a function satisfying \begin{equation}\ell(x) = x\quad \text{ for every }x \in M; \label{eq:def_of_blocks}\end{equation} the blocks in the partition are then the sets of form $\ell^{-1}(x)$, for $x \in M$.

\begin{definition}\label{def_merger_of_marked_partitions}
Assume $\Pi = (M,\ell)$ is a marked partition of $A$ and $x,y\in M$ are distinct marks. The partition $\Pi' = (M', \ell')$ of $A$ obtained by merging the blocks of $x$ and $y$ in $\Pi$ is defined as follows. Assume first that the cardinalities of $\ell^{-1}(x)$ and $\ell^{-1}(y)$ have different parities, and (without loss of generality) that $|\ell^{-1}(x)|$ is odd. We then let $M'=M\backslash \{y\}$ and 
 $$\ell'(z) = \begin{cases} x&\text{if } \; \; \ell(z) = y\\\ell(z)&\text{otherwise.}\end{cases}$$
 In words, the blocks of $x$ and $y$ are merged and $x$ is set as the mark of the resulting block. In case the cardinalities of $\ell^{-1}(x)$ and $\ell^{-1}(y)$ have the same parity, we elect one of $x$ and $y$ according to some arbitrary procedure (for example, the smaller one w.r.t.\ the order $\prec$ on $\Z^d$) and repeat the above definition, merging the blocks and making the elected point the new mark. 
\end{definition}

Assume given a probability measure $\mathbb{P}$ under which independent, $R$-spread-out random walk trajectories $\{(X^x_t)_{t\geq 0}: x \in A\}$ are defined, with $X^x_0 = x$ for each $x$. A construction of a system of coalescing random walks $\{(Y^x_t)_{t\geq 0}: x \in A\}$ will now be exhibited. The construction will rely on an auxiliary process $\{\Pi_t = (M_t,\ell_t): t \geq 0\}$ of marked partitions of $A$; once this auxiliary process is defined, we will simply set
\begin{equation}\label{eq:indep_to_coalesc}Y^x_t = X^{\ell_t(x)}_t,\quad x \in A,\; t \geq 0.\end{equation}

The definition of $\{\Pi_t: t \geq 0\}$ will be recursive.

\begin{definition}\label{def_coal_rws}
 Let $\Pi_0 = (M_0, \ell_0)$ be the trivial partition given by $M_0 = A$ and $\ell_0(x) = x$ for each $x \in A$. Also define $T_0 = 0$. Now assume that we have defined a stopping time $T_n$ (with respect to the filtration of the random walks) and also that we have defined $\{\Pi_t: 0 \leq t < \infty\}$ on $\{T_n = \infty\}$ and $\{\Pi_t: 0 \leq t \leq T_n\}$ on $\{T_n < \infty\}$. If $T_n = \infty$,  set $T_{n+1} = \infty$; otherwise let
$$T_{n+1} = \inf\{\; t > T_n:\;X^x_t = X^y_t \text{ for some distinct } x,y \in M_{T_n} \; \}.$$
For $T_n \leq t < T_{n+1}$, we set $\Pi_t = \Pi_{T_n}$. If $T_{n+1} < \infty$, then there exists a unique pair of distinct $x,y\in M_{T_n}$ such that $X^x_{T_{n+1}} = X^y_{T_{n+1}}$. We then let $\Pi_{T_{n+1}}$ be the marked partition obtained from $\Pi_{T_n}$ by merging the blocks of $x$ and $y$ according to Definition \ref{def_merger_of_marked_partitions}.
\end{definition}

It is now easy to see that Definition \ref{def_coal_rws} and \eqref{eq:indep_to_coalesc} produce a system of coalescing random walks and we omit the proof of this statement. Since the set of all intersection times of all the random walks $\{(X^x_t)_{t\geq 0}: x\in A\}$ is finite, it almost surely holds that $T_n = \infty$ for some (random) large enough $n$. It thus makes sense to define a ``terminal'' marked partition $\Pi_\infty = (M_\infty, \ell_\infty)$ given by $\Pi_t$ for $t$ large enough.

\medskip

We now show how the marked partition process also allows for the definition of a system $\{(\widetilde{Y}^x_t)_{t\geq 0}: x \in \widetilde{M}_t\}$ of \emph{annihilating random walks}
on the same probability space.

\begin{definition}\label{def_annih}
Let $\widetilde{M}_0=A$ and 
\begin{equation}\label{annih_index_set}
\widetilde{M}_t=\{\, x \in M_t \, : \, |\ell_t^{-1}(x)| \text{ is odd }  \},
\qquad t \in [0,\infty],
\end{equation}
then set
\begin{equation}\label{eq:indep_to_annih}
\widetilde{Y}^x_t = X^{x}_t,\quad x \in \widetilde{M}_t,\; t \in [0,\infty).
\end{equation}
\end{definition}
It is easy to check that Definition \ref{def_annih} indeed produces
a system of annihilating random walks,  that is, a system in which walkers perform independent continuous-time simple random walks on $\Z^d$ until two of them meet, and when they do, they immediately annihilate each other. 

Note that Definition \ref{def_annih} is the reason why parity played an important role in the way we defined the marked partition process in Definitions \ref{def_merger_of_marked_partitions} and \ref{def_coal_rws}. 

\medskip

The next lemma, which already appeared in \cite[Section 5]{RV15}, will be used to show that the number $|A\setminus M_{\infty}|$ of coalescences  is ``small'' if the range $R$ of random walk jumps is ``big''.
Recall that  $\prec$ denotes a well-ordering of $\Z^d$.
\begin{lemma}\label{lem:exponential_eta}
For any $\beta \in (0,1)$,
\begin{equation}\label{eq:exponential_eta}
\E\left[\beta^{-|A\setminus M_{\infty}|} \right]  \leq 
\prod_{x,y \in A,\;x \prec y} \left(1+h_R(x,y)\cdot (\beta^{-2} - 1) \right).
\end{equation}
\end{lemma}
\begin{proof}
For each distinct $x,y \in A$, let $\eta_{x,y}$ be the indicator of the event that, at some point in the marked partition process $\{\Pi_t: t \geq 0\}$, a block with mark $x$ and odd cardinality is merged with a block with mark $y$ and odd cardinality, i.e., the walkers $\widetilde{Y}^x$ and $\widetilde{Y}^y$ annihilate each other
before any other walker annihilates either of them.

Now we note that 
$\widetilde{M}_\infty \subseteq M_\infty$ and $|A\setminus \widetilde{M}_\infty|=2\mathcal{A}_\infty(A)$, where
$\mathcal{A}_\infty(A):=
\sum_{x, y \in A,\, x \prec y}\eta_{x,y}$,
thus we only need to show that $\E\left[\beta^{-2\mathcal{A}_\infty(A)} \right]$ is less than or equal to 
the right-hand side of \eqref{eq:exponential_eta} in order to conclude the proof of \eqref{eq:exponential_eta}. 
Now this is exactly \cite[(5.14)]{RV15}, thus the proof of Lemma
\ref{lem:exponential_eta} is complete. \end{proof}


\subsection{A bound on the probability of the joint occurrence of events}

 The aim of this lemma is to control positive correlations present in $\mu_{\alpha,R}$ using Lemma
\ref{lem:exponential_eta}.
 Recall the notion of $h_R(x,y)$ of from \eqref{eq:def_h}.

\begin{lemma}\label{lem:disjointly}
Let $B \subset \Z^d$ be finite with $0 \in B$ and $x_1,\ldots, x_n \in \Z^d$ be such that the sets
 $x_i + B$, $1\leq i \leq n$ are disjoint. Then, for any $E \in \mathcal{F}_B$,
\begin{equation}\label{disjointly_eq}
\mu_{\alpha,R}\left(\bigcap_{i=1}^n \theta_{x_i}E \right) 
\leq 
\pi_\alpha(E)^n\cdot \prod_{\substack{u,v \in \cup_i(x_i + B)\\u \prec v}} \left(1 + h_R(u,v) \cdot \left(\pi_\alpha(E)^{-2} -1\right) \right).
\end{equation}
\end{lemma}
\begin{proof} Let $B_i=x_i + B$ and
 $A = \cup_{i=1}^n B_i$. Let $\P$ be a probability measure under which independent, $R$-spread-out random walks $\{(X^x_t)_{t \geq 0}: x\in A\}$ with $X^x_0 = x$ are defined; let $\{(Y^x_t)_{t \geq 0}: x\in A\}$, $\Pi_t = (M_t,\ell_t)$ for $0 \leq t \le \infty$ be as defined in Section \ref{section:coal}. Also assume that under $\P$, and independently from the coalescing walks, independent Bernoulli($\alpha$) random variables $\{\zeta(x):x\in A\}$ are defined. Then set
\begin{equation}
\xi(x) = \zeta(\ell_\infty(x)),\quad x \in A.\label{eq:def_xi_zeta}
\end{equation}
Thus defined, the distribution of $\{\xi(x):x\in A\}$ is equal to $\mu_{\alpha,R}$ projected to $\{0,1\}^A$
(see \cite[Section 3]{RV15} for the details of this construction of $\mu_{\alpha,R}$).
Now define
\begin{equation}\label{def_I}
\mathcal{I} = \{ \, i: B_i \subset M_\infty \, \}.
\end{equation}
By \eqref{eq:def_of_blocks} and \eqref{eq:def_xi_zeta},
\begin{equation}\label{eq:where_zeta_xi}
i \in \mathcal{I} \quad \Longrightarrow \quad \xi(x) = \zeta(x) \; \text{ for all } x \in B_i.
\end{equation}

We now compute
\begin{multline}
\label{eq:last_line_long}
 \mu_{\alpha,R}\left(\bigcap_{i=1}^n \theta_{x_i}E \right) 
 =
\P\left[\xi \in \bigcap_{i=1}^n \theta_{x_i} E \right]
\leq 
\P\left[\xi \in \bigcap_{i \in \mathcal{I}} \theta_{x_i} E \right]
\stackrel{\eqref{eq:where_zeta_xi}}{=}
\P\left[\zeta \in \bigcap_{i \in \mathcal{I}} \theta_{x_i} E \right] \\
=
 \sum_{I \subset \{1,\ldots, n\}} \P[\,\mathcal{I} = I\,] \cdot \P\left[\left.\zeta \in \bigcap_{i \in I} \theta_{x_i} E \, \right| \, \mathcal{I} = I\right]
= \sum_{I \subset \{1,\ldots, n\}} \P[\, \mathcal{I} = I \, ] \cdot \pi_\alpha(E)^{|I|}\\
= \pi_\alpha(E)^n \cdot \E\left[\left(\frac{1}{\pi_\alpha(E)} \right)^{n-|\mathcal{I}|} \right].
\end{multline}
Now we note that
$$n - |\mathcal{I}| \stackrel{ \eqref{def_I} }{=} \sum_{i=1}^n \mathds{1}_{\{B_i \backslash M_\infty \neq \varnothing \}} 
\leq
 \sum_{i=1}^n |B_i \backslash M_\infty| = |A \backslash M_\infty|  $$
Finally, \eqref{disjointly_eq}  is obtained by plugging this inequality in \eqref{eq:last_line_long} and then applying Lemma \ref{lem:exponential_eta}.
\end{proof}

\subsection{Renormalization scheme}\label{ss:renormalization}
Together with Lemma \ref{lem:disjointly}, the main tool in our proof of Theorem \ref{thm:main} is multi-scale renormalization. Specifically, we will use the same renormalization scheme as in \cite{Ra15} and \cite{RV15}, which in turn is a variant of the one in \cite{sznitman_decoupling}. As in these references, renormalization is the ingredient which allows us to argue that large-scale percolation crossing events imply  numerous crossings of small and sparsely-located boxes.

Fix $d \geq 3$ and $L \in \N$. Define
$$L_N = 6^N\cdot L,\quad \mathcal{L}_N = L_N \cdot \Z^d,\qquad N \geq 0.$$
For $k \geq 0$, let $T_{(k)} = \{1,2\}^k$ (with $T_{(0)} = \{\emptyset\}$) and let $$T_N = \bigcup_{k=0}^N T_{(k)}$$ be  the binary tree of height $N$. For $0 \leq k < N$ and $m = (\eta_1,\ldots, \eta_k) \in T_{(k)}$, let $m_1= (\eta_1,\ldots, \eta_k,1)$ and $m_2 = (\eta_1,\ldots, \eta_k,2)$ be the two children of $m$.
\begin{definition}
$\mathcal{T}: T_N \to \Z^d$ is a proper embedding of $T_N$ if
\begin{enumerate}
\item $\mathcal{T}(\{\emptyset\}) = 0;$ 
\item for all $0 \leq k \leq N$ and $m \in T_{(k)}$ we have $\mathcal{T}(m) \in \mathcal{L}_{N-k}$;
\item for all $0 \leq k < N$ and $m \in T_{(k)}$ we have
\begin{equation}
|\mathcal{T}(m_1) - \mathcal{T}(m)| = L_{N-k},\qquad |\mathcal{T}(m_2) - \mathcal{T}(m)| = 2L_{N-k}.
\end{equation}
\end{enumerate}
We let $\Lambda_N$ denote the set of proper embeddings of $T_N$ into $\Z^d$.
\end{definition}

We will now reproduce three results concerning proper embeddings. Their proofs can be found in \cite{Ra15}. The first of them bounds the number of proper embeddings.
\begin{lemma}
\label{lem:renorm_count} There exists $C_d > 0$ such that, for all $N \in \N$,
\begin{equation} 
\label{eq:renorm_count}
|\Lambda_N| \leq (C_d)^{2^N}.
\end{equation}
\end{lemma}
The second result establishes a relation between proper embeddings and crossing events. For a helpful illustration, see Figure 2 in \cite{RV15}.
\begin{lemma}\label{lem:renorm_paths}
If $\gamma$ is a $*$-connected path in $\Z^d$ with
$$\{\gamma\} \cap S(L_N-1) \neq \varnothing, \quad \{\gamma\} \cap S(2L_N) \neq \varnothing,$$
then there exists $\mathcal{T} \in \Lambda_N$ such that
\begin{equation}
\label{eq:renorm_paths} \{\gamma\} \cap S(\mathcal{T}(m),L_0-1) \neq \varnothing,\;\; \{\gamma\} \cap S(\mathcal{T}(m),2L_0) \neq \varnothing \quad \forall m \in T_{(N)}.
\end{equation}
\end{lemma}
Finally, the third result guarantees that, in a proper embedding, the `bottom-level' boxes are sparsely located.
\begin{lemma}\label{lem:positions}
For any $\mathcal{T} \in \Lambda_N$ and any $m_0 \in T_{(N)}$, 
\begin{equation}
\label{eq:positions} \left| \left\{ m \in T_{(N)}: \mathrm{dist}\left(B(\mathcal{T}(m_0),2L),B(\mathcal{T}(m),2L) \right) \leq 6^k\cdot L/2\right\} \right| \leq 2^{k-1},\quad k\geq 1.
\end{equation}
\end{lemma}

\section{Proof of Theorem \ref{thm:main}}

Recall from the Introduction that
$$\text{Perc} = \left\{\begin{array}{l}\xi \in\{0,1\}^{\Z^d}: \text{ there exists an infinite nearest-neighbor  }\\ \text{path $\gamma = (\gamma(0), \gamma(1),\ldots)$ such that $\xi(\gamma(i)) = 1$ for each $i$}  \end{array} \right\}.$$
\subsection{Absence of percolation for $\alpha < p_c(\Z^d)$ and $R$ large}\label{ss:subcritical}
The goal of this subsection is establishing that \begin{equation}\label{eq:liminf_res}\liminf_{R\to\infty} \alpha_c(R) \geq p_c.\end{equation}
Fix $\alpha < p_c$. It will be shown that there exist $L$ and $R_0$ in $\N$ such that, letting $L_N = L\cdot 6^N$ and for any $R \geq R_0$,
\begin{equation}
\label{eq:connect_sub}
\mu_{\alpha,R}\left[B(L_N-1) \stackrel{\xi}{\longleftrightarrow} B(L_N) \right] < 2^{-2^N},\quad N \in \mathbb{N}.
\end{equation}

Observing that $$\text{Perc} \subseteq \bigcup_{M\geq 1}\bigcap_{N \geq M} \{B(L_N-1) \stackrel{\xi}{\longleftrightarrow} B(L_N)\},$$
one notes that if \eqref{eq:connect_sub} holds, then $\mu_{\alpha,R}(\text{Perc}) = 0$, that is, $\alpha_c(R) \geq \alpha$, so \eqref{eq:liminf_res} follows.

Since $\alpha < p_c$, it is possible to choose  (and fix) $L$ large enough (depending on $\alpha$) such that
\begin{equation}
\pi_\alpha\left[B(L) \stackrel{\xi}{\longleftrightarrow} B(2L)^c \right] < (4C_d)^{-1},
\label{eq:choice_L}\end{equation}
where $C_d$ is as in Lemma \ref{lem:renorm_count}.
 This is a simple consequence of the exponential decay of the cluster radius beneath $p_c$; see for instance Section 5.2 in \cite{Gr99}.

Now, Lemma \ref{lem:renorm_paths}, a union bound and Lemma \ref{lem:renorm_count} give
\begin{align}\nonumber
&\mu_{\alpha,R}\left[B(L_N-1) \stackrel{\xi}{\longleftrightarrow} B(2L_N)^c \right] \\& \leq (C_d)^{2^N}\cdot \max_{\mathcal{T} \in \Lambda_N} \mu_{\alpha,R}\left(\bigcap_{m \in T_{(N)}} \{B(\mathcal{T}(m),L)\stackrel{\xi}{\longleftrightarrow} B(\mathcal{T}(m),2L)^c \}\right).\label{eq:comp_bound}\end{align}
In order to bound the maximum on the right-hand side, fix $\mathcal{T} \in \Lambda_N$. Define the event
$$E = \left\{ B(0,L)\stackrel{\xi}{\longleftrightarrow} B(0,2L)^c\right\}.$$ 
Also define the set
\begin{equation}\label{A_union_balls}
A = \bigcup_{m\in T_{(N)}} B(\mathcal{T}(m),2L).
\end{equation}
Note that by \eqref{eq:positions}, the balls in the above union are disjoint. Then, Lemma \ref{lem:disjointly} gives
\begin{align}\nonumber
&\mu_{\alpha,R}\left(\bigcap_{m \in T_{(N)}} \{B(\mathcal{T}(m),L)\stackrel{\xi}{\longleftrightarrow} B(\mathcal{T}(m),2L)^c \}\right)\\
\nonumber&\leq (\pi_\alpha(E))^{2^N} \cdot \prod_{\substack{u,v \in A,\; u \prec v}} \left(1+ h_R(u,v)\cdot \left( \pi_\alpha(E)^{-2} - 1\right) \right)\\
&\stackrel{\eqref{eq:bound_green},\eqref{eq:choice_L}}{\leq} (4C_d)^{-2^N}\cdot 
\exp\left\{ \left(\pi_\alpha(E)^{-2} - 1\right)f(R) \sum_{\substack{u,v\in A,\; u \prec v}}|u-v|^{2-d} \right\}.\label{eq:comp_bound2}
\end{align}

Now, putting together \eqref{eq:comp_bound} and \eqref{eq:comp_bound2} and recalling
from \eqref{eq:bound_green} that $f(R) \xrightarrow{R\to\infty} 0$, the desired convergence \eqref{eq:connect_sub} will follow from showing that
\begin{equation}
\label{eq:upper_A}
\sum_{\substack{u,v\in A,\; u \prec v}}|u-v|^{2-d}\leq C' 2^N,
\end{equation}

where $C'$ is a positive constant that does not depend on $N \in \N$ or $\mathcal{T} \in \Lambda_N$. To this end, define 
\begin{equation}\label{V_u_k_set}
\mathcal{V}_\mathcal{T}(u,k) = \left\{ v \in A:\; v \neq u,\; |u-v| \leq 6^k L/2 \right\}, \qquad u \in A,\;k \geq 1. 
\end{equation}
Now  we have
\begin{multline}
\label{cardinality_bound}
|\mathcal{V}_\mathcal{T}(u,k)| \stackrel{\eqref{A_union_balls}}{\leq} |B(0,2L)|\cdot
\left| \left\{ m \in T_{(N)}: \mathrm{dist}\left(u ,B(\mathcal{T}(m),2L) \right) \leq 6^k L/2\right\} \right| \stackrel{ \eqref{eq:positions}  }{\leq} \\
|B(0,2L)|\cdot 2^{k-1} 
 = (4L + 1)^d \cdot  2^{k-1} \qquad \text{for any} \quad k \geq 1.
\end{multline}
Hence, for all $u\in A$,
\begin{multline}\label{C_prime_prime}
\sum_{\substack{v\in A,\; v \neq u}}|u-v|^{2-d}=
\sum_{v \in \mathcal{V}_\mathcal{T}(u,1)}|u-v|^{2-d}+\sum_{k=1}^{\infty} \;
\sum_{v \in \mathcal{V}_\mathcal{T}(u,k+1)\setminus \mathcal{V}_\mathcal{T}(u,k)  }|u-v|^{2-d}
\stackrel{ \eqref{V_u_k_set} }{\leq} \\
C+\sum_{k=1}^\infty \;
\sum_{v \in \mathcal{V}_\mathcal{T}(u,k+1)\setminus \mathcal{V}_\mathcal{T}(u,k)  } (6^kL/2)^{2-d}
 \stackrel{ \eqref{cardinality_bound} }{\leq}
C+ \sum_{k=1}^\infty (4L + 1)^d \cdot  2^{k} \cdot (6^kL/2)^{2-d} \leq
  C''
\end{multline}
and thus \eqref{eq:upper_A} follows:
\begin{equation*}
\sum_{\substack{u,v\in A,\; u \prec v}}|u-v|^{2-d} 
\stackrel{ \eqref{C_prime_prime} }{\leq} |A| \cdot C'' 
\stackrel{ \eqref{A_union_balls} }{=}
 2^N\cdot (4L + 1)^{d} \cdot C'' = 2^N C'.
\end{equation*}

\subsection{Percolation for $\alpha > p_c(\Z^d)$ and $R$ large}
It will now be shown that 
\begin{equation}\label{limsup_leq_p_c}
\limsup_{R \to \infty} \alpha_c(R) \leq p_c.
\end{equation}
To this end, fix $\alpha > p_c$; it suffices to show that, if $R$ is large enough, there exists an infinite percolation cluster with probability 1 under $\mu_{\alpha, R}$.

The proof involves a two-step renormalization scheme. 

The first step is a coarse graining of the lattice in which the configuration $\xi$ sampled from $\mu_{\alpha, R}$ is defined. More specifically, for some large $M \in \N$ and each $x \in \Z^d$,
we define the box $B(M x, M)$  associated to vertex $x$ in the renormalized lattice. 
A configuration $\tilde \xi \in \{0,1\}^{\Z^d}$ is then defined in the renormalized lattice through the prescription that $\tilde \xi(x) = 1$ when the restriction of $\xi$ to the box $B(M x, M)$ contains
a ``special'' locally unique giant connected cluster of $\xi$-open sites, see \eqref{high_connect}.
By \eqref{high_connect},  neighbouring ``special'' clusters are connected to each other, thus an infinite  $\tilde \xi$-open cluster guarantees the existence of an infinite open cluster for $\xi$ as well.

In the second step, the goal is to prove that $\tilde \xi$ indeed contains an infinite cluster for some large $M$ and any $R$ larger than some $R_0$. This is done through a renormalization of the type described in Section \ref{ss:renormalization} in the lattice in which $\tilde \xi$ is defined.

The starting point is defining the `locally supercritical' property involved in the first renormalization step. Given $M \in \N$, let $E(M)$ be the set of configurations $\xi \in \{0,1\}^{\Z^d}$ satisfying:
\begin{equation}\label{high_connect}
E_M = \left\{\begin{array}{c}\xi \in \{0,1\}^{\Z^d}:  \xi|_{B(M)}\text{ has a unique open cluster of }\\
\text{  diameter greater than or equal to $M$, moreover } \\
\text{ this cluster intersects all the faces of $B(M)$}\end{array} \right\}.
\end{equation}
Note that $E_M \in \mathcal{F}_{B(M)}$. If $E_M$ occurs, we call the unique cluster that appears in 
\eqref{high_connect} the \emph{special} cluster of $B(M)$. Then, given $\xi \in \{0,1\}^{\Z^d}$, define
\begin{equation}\label{eq:def_of_tilde_xi}
\tilde \xi(x) = \mathds{1}_{\theta_{Mx}E_M},\quad x \in \Z^d.
\end{equation}
Note that, if $x,y \in \Z^d$ with $|x-y|_1 = 1$ and $\tilde \xi(x)=\tilde \xi(y)=1$,
 then the special cluster of $\xi|_{B(x,M)}$ necessarily intersects the special cluster 
 of $\xi|_{B(y,M)}$. This consideration leads to the conclusion that if there exists an infinite nearest neighbor path $\tilde \upgamma$ of open sites in $\tilde \xi$, then there exists an infinite nearest neighbor path $\upgamma$ of open sites in $\xi$. Hence, letting $\tilde \mu_{\alpha, R,M}$ denote the distribution of $\tilde \xi$ when $\xi$ is sampled from $\mu_{\alpha, R}$,
\begin{equation}\label{eq:mu_tilde_mu}
\mu_{\alpha, R}(\text{Perc}) \geq \tilde \mu_{\alpha,R,M}(\text{Perc}).
\end{equation}

We now claim that, for supercritical Bernoulli site percolation, the event $E_M$ is very likely when $M$ is large:
\begin{equation}\label{eq:connectivity_event}
\alpha > p_c \quad \Longrightarrow \quad \lim_{M \to \infty} \pi_\alpha(E_M) = 1.
\end{equation}
The analogous statement for supercritical \textit{bond} percolation on $\Z^d$ is Theorem (7.61) in \cite{Gr99}. The proof is based on a block argument originally developed in \cite{Pi96} and \cite{DP96} which, as mentioned in the latter reference, works equally well for site percolation. We thus omit the proof of \eqref{eq:connectivity_event}.
By \eqref{eq:connectivity_event}, we can  find (and fix) $M$ such that 
\begin{equation}\label{eq:choice_of_M}
\pi_\alpha(E_M) >1- (4C_d)^{-1},
\end{equation}
where $C_d$ is the constant of Lemma \ref{lem:renorm_count}.

Now consider the renormalization scheme of Section \ref{ss:renormalization} with $L_0 = 1$ (so that, for $N \in \N$, $L_N = 6^N$). It will be shown that there exists $R_0$ such that if $R \geq R_0$ then
\begin{equation}
\label{eq:red_second_step}
\tilde \mu_{\alpha, R,M}\left[ B(L_N-1) \stackrel{*(1-\tilde\xi)}{\longleftrightarrow} B(L_N) \right] < 2^{-2^N},\quad N \in \N.
\end{equation}
In words, the probability that there is a $*$-connected path of \textit{closed} sites in $\tilde \xi$ connecting $B(L_N-1)$ to the outside of $B(L_N)$ is smaller than $2^{-2^N}$. Standard considerations involving planar duality (i.e., a Peierls argument) show that \eqref{eq:red_second_step} implies 
\begin{equation}\label{eq:perc_tilde}
\tilde \mu_{\alpha, R,M}(\text{Perc}) = 1.
\end{equation}
We refer readers who are unfamiliar with this type of proof to Section 4.1 of \cite{RV15}, where the same line of reasoning is carried out in detail.  Finally, together with \eqref{eq:mu_tilde_mu}, \eqref{eq:perc_tilde} yields the desired result $\mu_{\alpha, R}(\text{Perc}) = 1$.

It remains to prove \eqref{eq:red_second_step}. Lemma \ref{lem:renorm_paths}, a union bound and Lemma \ref{lem:renorm_count} give, for any $N \in \N$,
\begin{equation}\label{eq:comp_bound_tilde}
\tilde\mu_{\alpha,R,M}\left[B(L_N-1) \stackrel{*(1-\tilde\xi)}{\longleftrightarrow} B(2L_N)^c \right] 
 \leq (C_d)^{2^N}\cdot \max_{\mathcal{T} \in \Lambda_N} \tilde\mu_{\alpha,R,M}\left(\bigcap_{m \in T_{(N)}} \{\tilde \xi(\mathcal{T}(m))=0\}\right).
\end{equation}
Now, for any $\mathcal{T} \in \Lambda_N$,
\begin{equation}\label{switching_from_tilde}
\tilde\mu_{\alpha,R,M}\left(\bigcap_{m \in T_{(N)}} \{\tilde \xi(\mathcal{T}(m))=0\}\right) \stackrel{\eqref{eq:def_of_tilde_xi}}{=} \mu_{\alpha, R}\left(\bigcap_{m \in T_{(N)}}( \theta_{M\cdot \mathcal{T}(m)}E_M )^c\right).
\end{equation}
Noting that the sets $B(M\cdot \mathcal{T}(m),M)$ for $m \in T_{(N)}$ are pairwise disjoint and defining $A$ as the union of all these sets, Lemma \ref{lem:disjointly} can then be applied as in \eqref{eq:comp_bound2}, yielding
\begin{equation}\label{supcrit_bound_decorr_applied}
\mu_{\alpha, R}\left(\bigcap_{m \in T_{(N)}}( \theta_{M\cdot \mathcal{T}(m)}E_M )^c\right) 
< (4C_d)^{-2^N} \exp\left\{\left(\pi_\alpha(E_M)^{-2} - 1\right)f(R) 
\sum_{\substack{u,v\in A,\; u \prec v}}|u-v|^{2-d} \right\}. 
\end{equation}
Notice the similarity between \eqref{supcrit_bound_decorr_applied} and \eqref{eq:comp_bound2}.
Now \eqref{eq:red_second_step} follows  from \eqref{supcrit_bound_decorr_applied} 
using the same calculations that we used to show that \eqref{eq:connect_sub} follows from \eqref{eq:comp_bound2}: a variant of \eqref{eq:upper_A} together with \eqref{eq:bound_green}, 
\eqref{eq:comp_bound_tilde}--\eqref{supcrit_bound_decorr_applied} gives \eqref{eq:red_second_step}, completing
the proof of \eqref{limsup_leq_p_c}.

\medskip 
 
 {\bf Acknowledgements:} We thank G\'abor Pete for valuable discussions.
  The work of B.R.\ is partially supported by OTKA (Hungarian
National Research Fund) grants K100473 and K109684, the Postdoctoral Fellowship of
NKFI (National Research, Development and Innovation Office) and the Bolyai Research Scholarship
of the Hungarian Academy of Sciences.


{\footnotesize
\begin{thebibliography}{99}

\bibitem[BGP]{bgp}
Benjamini, I., Gurel-Gurevich, O., and Peled, R.
\textit{On K-wise Independent Distributions and Boolean Functions}.
\emph{arXiv:1201.3261}

\bibitem[BLM87]{BLM87} Bricmont, J., Lebowitz, J. and Maes, C. \textit{Percolation in strongly correlated systems: the massless Gaussian field}. Journal of statistical physics 48, no. 5 (1987): 1249-1268.

\bibitem[CS73]{CS73} Clifford, P., and Sudbury, A. \textit{A model for spatial conflict}. Biometrika 60, no. 3 (1973): 581-588.

\bibitem[DP96]{DP96} Deuschel, J-D., and  Pisztora, A. \textit{Surface order large deviations for high-density percolation}. Probability Theory and Related Fields 104, no. 4 (1996): 467-482.

\bibitem[Gr99]{Gr99} Grimmett, G. \textit{Percolation}. Springer-Verlag Berlin (Second edition) (1999).

\bibitem[HL75]{HL75} Holley, R., and Liggett, T. \textit{Ergodic theorems for weakly interacting infinite systems and the voter model}. The Annals of probability (1975): 643-663.

\bibitem[LS88]{LS88} Lebowitz, J.\ L., and  Schonmann, R.\ H. 
\textit{Pseudo-free energies and Large deviations for Non Gibbsian FKG measures.}
 Probability Theory and Related Fields 77.1 (1988): 49-64.

\bibitem[LS86]{LS86} Lebowitz, J., and Saleur, H. \textit{Percolation in strongly correlated systems}. Physica A: Statistical Mechanics and its Applications 138, no. 1-2 (1986): 194-205.

\bibitem[Li85]{Li85} Liggett,  T. \textit{Interacting particle systems}. Grundlehren der mathematischen Wissenschaften 276, Springer (1985).

\bibitem[Ma07]{Ma07} Marinov, V. \textit{Percolation in correlated systems}. PhD Thesis, Rutgers The State University of New Jersey-New Brunswick, 2007.

\bibitem[ML06]{ML06} Marinov, V., and Lebowitz, J. \textit{Percolation in the harmonic crystal and voter model in three dimensions}. Physical Review E 74, no. 3 (2006): 031120.

\bibitem[Pi96]{Pi96} Pisztora, A. \textit{Surface order large deviations for Ising, Potts and percolation models}. Probability Theory and Related Fields 104, no. 4 (1996): 427-466.

\bibitem[PR15]{popov_rath}
 Popov, S., and R\'ath, B. 
\textit{On decoupling inequalities and percolation of excursion sets of the Gaussian free field.} 
J. of Stat. Phys., (2015),  159 (2), 312-320.

\bibitem[PT15]{popov_teixeira} Popov, S., and Teixeira, A.
\textit{Soft local times and decoupling of random interlacements.} 
  J. European Math. Soc. 17 (10), 2545-2593 (2015).

\bibitem[Ra15]{Ra15} R\'ath, B. \textit{A short proof of the phase transition for the vacant set of random interlacements}. Electronic Communications in Probability 20 (2015).

\bibitem[RV15]{RV15} R\'ath, B., and Valesin, D. \textit{Percolation on the stationary distributions of the voter model}. Annals of Probability  45 (3), 1899-1951 (2017).

\bibitem[RS13]{RS13}  Rodriguez, P.-F., and  Sznitman, A.-S. \textit{Phase transition and level set percolation for the Gaussian free field}. Communications in Mathematics Physics 320 (2), 571--601 (2013).

\bibitem[R17]{rodriguez17} Rodriguez, P.-F.,
\textit{Decoupling inequalities for the Ginzburg-Landau} $\nabla \varphi$ \textit{models.}
arXiv: 1612.02385.

\bibitem[ST17]{st17}
Steif, J., Tykesson, J.
\textit{Generalized Divide and Color models}. \emph{arXiv:1702.04296}


\bibitem[Sz12]{sznitman_decoupling}  Sznitman, A.-S.
\textit{Decoupling inequalities and interlacement percolation on $G \times \Z$}.
Inventiones mathematicae, 187, 3, 645-706  (2012).


\end{thebibliography} 

}
\end{document}